\documentclass{article}
\usepackage[utf8]{inputenc}
\usepackage{lipsum}
\usepackage[finnish,swedish,english]{babel}
\usepackage{url}
\usepackage{graphicx}
\usepackage{amsmath}
\usepackage{amsthm}
\usepackage{amssymb}
\usepackage{mathrsfs}
\usepackage{setspace}
\usepackage{color}

\newtheorem{theorem}{Theorem}[section]
\newtheorem{lemma}[theorem]{Lemma}
\newtheorem{proposition}[theorem]{Proposition}
\newtheorem{corollary}[theorem]{Corollary}

\newtheorem{definition}[theorem]{Definition}

\newtheorem{example}[theorem]{Example}

\newtheorem{assumption}[theorem]{Assumption}

\author{Erik Sj\"oland}
\title{Using real algebraic geometry to solve combinatorial problems with symmetries}

\begin{document}

\newpage

\maketitle

\begin{abstract}
Many combinatorial problems can be formulated as a polynomial optimization problem that can be solved by state-of-the-art methods in real algebraic geometry. In this paper we explain many important methods from real algebraic geometry, we review several applications and discuss implementation and computational aspects.
\end{abstract}

\tableofcontents

\newpage

\section{Introduction to semidefinite programing}
\label{sec:SDP}
Special cases of Linear Programming (LP) dates back to Fourier, but linear programming was in its generality first studied by Leonid Kantorovich in 1939 to reduce costs for the Russian army during World War II. In 1947 George Danzig published a paper with the famous \emph{simplex method} to solve LPs, which is usually extremely efficient in practice but has exponential worst case scenarios. Several other algorithms have been proposed including the \emph{ellipsoid algorithm}, which was proposed by Shor in 1972. Even though the ellipsoid algorithm is not convenient in practice it is the first algorithm proven to solve linear programming problems in polynomial time by Khachiyan in 1979 \cite{Khachiyan1979}. Another class of algorithms to solve linear programs are the \emph{interior point methods}. The original method was invented by John von Neumann around the same time as Danzig introduced the simplex algorithm, and was later popularized by an efficient algorithm to give an approximate solution to an LP problem by Narendra Karmarkar in 1984 \cite{Karmarkar1984}.  Complexity of LP-algorithms is still an active field of research, and LP has since it was introduced been used in a tremendous range of applications.

Semidefinite programming (SDP) started evolving from Linear Programming in the 1960s, and can be used to solve a wider variety of problems. Any LP can be formulated as an SDP, which is why it is natural to first define LP and then let it serve as a bridge to and as a first example of an SDP. There are various algorithms for finding close to optimal solutions to SDPs in polynomial time. Interior point methods have been the most popular and there are several SDP solvers (including CSDP, SeDuMi, SDPT3, DSDP, SDPA) that have good implementations of these algorithms. For a more extensive overview we refer to \cite{Vandenberghe1996, Todd2001}.

In an LP you want to minimize or maximize a linear function under linear constraints. Any LP can be rewritten in canonical form:
\[ 
\begin{array}{rl}
 \inf  & \displaystyle  c^Tx \\
\textnormal{subject to} & \displaystyle  Ax \leq b, \\
& \displaystyle  x \geq 0, \\
& \displaystyle  x \in \mathbb{R}^n,
\end{array}
\]
where $A \in \mathbb{R}^{m \times n}$, $b \in \mathbb{R}^m$, $c \in \mathbb{R}^n$. In many textbooks there are also equality constraints in the formulation of the LP, but these are redundant as they can be rewritten as inequalities by $d^Tx = e \Leftrightarrow \{ d^Tx \leq e \textrm{ and } -d^Tx \leq -e\}$.

An LP in canonical form is referred to as the \emph{primal problem}, and we define the \emph{dual problem} of an LP to be
\[ 
\begin{array}{rl}
\sup  & \displaystyle  b^Ty \\
\textnormal{subject to} & \displaystyle  A^Ty \geq c, \\
& \displaystyle  y \geq 0, \\
& \displaystyle  y \in \mathbb{R}^m.
\end{array}
\]
The dual is also a linear LP problem, and it plays a strong role of many of the algorithms for solving the original primal LP problem. Fundamental in duality theory is that the dual of the dual problem is the original primal problem. Even more important for optimization theory is the concepts of \emph{weak duality} and \emph{strong duality}. Weak duality ensures that any value of the dual problem gives a lower bound to the primal problem, and strong duality ensures that if the solution to one of the problems is finite the primal and dual solutions are equal. To prove that weak and strong duality holds for linear programs is a good exercise, and part of any elementary optimization course.

We are now ready to take the step into semidefinite programming. The essential idea is that we want to replace the non-negativity condition on the vector $x \in \mathbb{R}^n$ with an appropriate non-negativity condition on a symmetric matrix $X \in \mathbb{R}^{n \times n}$. It is natural to replace the non-negativity condition with a \emph{linear matrix inequality} (LMI), which is equivalent to requiring that the matrix $X$ is \emph{positive semidefinite}, which has several equivalent definitions:
\begin{theorem}
Let $A$ be a symmetric $n \times n$ matrix. The following are equivalent
\begin{itemize}
\item[(a)] $A$ is positive semidefinite.
\item[(b)] $x^TAx \geq 0$ for all $x \in \mathbb{R}^n$.
\item[(c)] All eigenvalues of $A$ are nonnegative.
\item[(d)] There is a unique lower triangular matrix $L \in \mathbb{R}^{n \times n}$ with $L_{jj} \geq 0$ for all $j \in \{1,\dots, n\}$ such that $LL^T = A$.
\item[(e)] All principal minors of $A$ are nonnegative.
\end{itemize}
\end{theorem}
The trace is a linear function on matrices, which generalizes the linear polynomials in LP. The canonical form of a semidefinite program is:
\[ 
\begin{array}{rll}
\inf  & \displaystyle  \mathrm{tr}(C^TX) \\
\textnormal{subject to} & \displaystyle  \mathrm{tr}(A_i^TX) \leq b_i, & 1 \leq i \leq m, \\
& \displaystyle  X \succeq 0, \\
& \displaystyle  X \in \mathbb{R}^{n\times n}, \\
& \displaystyle  X \textrm{ is symmetric},
\end{array}
\]
where $C \in \mathbb{R}^{n\times n}$, $A_i \in  \mathbb{R}^{n \times n}$ and $b_i \in \mathbb{R}$ for $1 \leq i \leq m$.

Observe that the feasible set $\{X \textrm{ symmetric real } n\times n \textrm{-matrix} :  \mathrm{tr}(A_i^TX) \leq b_i \textrm{ for } i =1,\dots,m, X \succeq 0\}$ is convex, and so SDP is a special class of convex optimization problems. 

Just as in the LP case the original problem is called primal problem, and there is an associated dual problem:
\[ 
\begin{array}{rll}
\sup  & \displaystyle  b^Ty \\
\textnormal{subject to} & \displaystyle  C-\sum_{i=1}^m A_iy_i \succeq 0, \\
& \displaystyle  y \in \mathbb{R}^{m}.
\end{array}
\]
Weak duality holds just as for LP, but strong duality does not always hold as the following standard example shows:
\begin{example}
The primal problem
\[ 
\begin{array}{rll}
\inf  & x_1 \\
\textnormal{subject to} & 
\displaystyle  \left( \begin{array}{ccc}
0 & x_1 & 0 \\
x_1 & x_2 & 0 \\
0 & 0 & x_1+1 
\end{array} \right) \succeq 0, \\
 & \displaystyle  x_1,x_2 \in \mathbb{R}, \\
\end{array}
\]
has the dual
\[ 
\begin{array}{rll}
\sup  & \displaystyle  -y_2 \\
\textnormal{subject to} & 
 \displaystyle  \left( \begin{array}{ccc}
y_1 & (1-y_2)/2 & 0 \\
(1-y_2)/2 & 0 & 0 \\
0 & 0 & y_2 
\end{array} \right) \succeq 0, \\
 & \displaystyle  y_1,y_2 \in \mathbb{R}. \\
\end{array}
\]
In order for the matrices to be positive semidefinite we see that both $x_1$ and $(1-y_2)/2$ must be $0$, so the solution to the primal problem is $0$ and the solution to the dual problem is $-1$.
\end{example}
Even though strong duality does not always hold it turns out that it holds under a condition called Slater's condition, which is a requirement of positive definiteness in either the primal or dual. Because of the importance of strong duality we finish the introduction by giving the exact statement:
\begin{theorem}
Let (P) denote a primal problem in canonical form and (D) its dual. Let $p^*$ and $d^*$ denote the optimal values of (P) and (D) respectively. If $p^*$ is finite and there exists a solution $X$ to (P) with $X \succ 0$, then (D) is feasible and $p^*=d^*$. Analogously, if $d^*$ is finite and there exists a solution $y$ to (D) for which $C - \sum_{i=1}^m A_iy_i \succ 0$, then (P) is feasible and $p^*=d^*$.
\end{theorem}

\newpage
\section{Polynomial optimization using real algebraic geometry}
\label{sec:RAG}
In both mathematical and real world applications we encounter problems where we need to find the optimal value of a polynomial under polynomial constraints:
\begin{equation} 
\label{eq:prim0}
\begin{array}{rll}
\inf  & \displaystyle  f(x) \\
\textnormal{subject to} & \displaystyle  g_1(x) \geq 0, \dots, g_m(x) \geq 0, \\
& \displaystyle  x \in \mathbb{R}^n.
\end{array}
\end{equation}
We discuss how to attack polynomial optimization problems using methods that have grown out from classical questions about sums of squares. Although some of the problems and results in this chapter dates back to Hilbert it is still not obvious how to find exact or approximative solutions to most polynomial optimization problem. This is still an active area of research, and there are still many interesting related open problems. We aim to keep the level of abstraction low in order to make it easier to follow, and we only discuss results in real algebraic geometry related to polynomial optimization. A more extensive survey is the chapter on the subject by Laurent \cite{Laurent2009}. It is important to note that once a final sum of squares certificate has been found using our proposed methods, one does not need polynomial optimization to verify that the solution is correct.

\begin{definition}
Let $f_1,\dots,f_s,g \in \mathbb{R}[x_1,\dots,x_n]$.

A subset of $\mathbb{R}^n$ is \emph{basic semialgebraic} if it is on the form $\{ f_1 \geq 0, \dots, f_s \geq 0, g \neq 0 \}$ and \emph{semialgebraic} if it is a finite union of basic semialgebraic sets.

A subset of $\mathbb{R}^n$ is \emph{basic closed semialgebraic} if it is on the form $\{ f_1 \geq 0, \dots, f_s \geq 0\}$.
\end{definition}
We mentioned earlier that we can state equalities as inequalities, $f = 0 \Leftrightarrow f \geq 0, -f \geq 0$, and in a similar fashion we can replace strict inequalities: $ f> 0 \Leftrightarrow f \geq 0, f \neq 0$. We see that the basic semialgebraic sets contains all sets defined by $=, \geq, \neq, >$, whereas the basic closed semialgebraic sets are limited to $=$ and $\geq$, and is more suitable for optimization.

The classical Nullstellensatz that dates back to Hilbert states that if $f \in \mathbb{C}[x_1,\dots,x_n]$ vanishes on the subset of $\mathbb{C}^n$ defined by $g_1=g_2=\dots=g_m=0$, where $g_1,\dots,g_m \in \mathbb{C}[x_1,\dots,x_n]$, then some power of $f$ lies in the ideal of $\mathbb{C}[x_1,\dots,x_n]$ generated by $g_1,\dots,g_m$. The Nullstellensatz is a fundamental result in classical algebraic geometry, which deals with subsets in $\mathbb{C}^n$ defined by polynomial equations. Real algebraic geometry deals with subset of $\mathbb{R}^n$ defined by polynomial equations and inequalities, and the two topics are in some respects similar and in others not. An essential property for the field $\mathbb{C}^n$ is that it is algebraically closed; every non-constant polynomial $f \in \mathbb{C}[x_1,\dots,x_n]$ has a root in $\mathbb{C}^n$. This is not the case for $\mathbb{R}^n$ as for example $x^2+1=0$ does not have any real solutions. Despite this major drawback for the real numbers, many results carry over from the analogous complex cases and similar methods and techniques can often be used in real algebraic geometry. One thing that is not immediately clear is how to translate the Nullstellensatz into the real algebraic context, and a lot of research has been done in this direction. The analogue to the Nullstellensatz is called the Positivstellensatz, and there are several versions of it depending on, among other things, the properties of the set $\{g_1 \geq 0, \dots, g_m \geq 0 \}$, and they are all aiming to find the most suitable condition to guarantee positivity of $f$. Once we have a certificate of positivity we can find the minimum of a polynomial $f(x) \in \mathbb{R}[x]$ by finding the largest $\lambda \in \mathbb{R}$ such that $f(x)-\lambda$ is positive for all $x \in \mathbb{R}^n$. In this section we explore the duality between positive polynomials and the theory of moments, and how it can be used to find solutions or bounds to polynomial optimization problems.

\section{Sums of squares and the Positivstellensatz}
\label{sec:Pos}
To be able to state the form of the Positivstellensatz that is most suitable for our optimization purposes we first need to understand the basic connections between positivity and sums of squares. First some simplifying notation: We write $\mathbb{R}[x]=\mathbb{R}[x_1,\dots,x_n]$, $f \geq 0$ when $f(x) \geq 0$ for all $x \in \mathbb{R}^n$ and $f > 0$ when $f > 0$ for all $x \in \mathbb{R}^n$.

It is obvious that if $f$ is a sum of squares, say $f=f_1^2+\dots+f_k^2$, then $f(x) = f_1(x)^2+\dots+f_k(x)^2 \geq 0$ for all $x \in \mathbb{R}^n$ whereas questions about the converse can be very difficult. There exist nonnegative polynomials that are not sums of squares of polynomials, which was known already by Hilbert. Motzkin provided the first concrete example of such a polynomial in 1967: 
\begin{proposition}
The polynomial $s(x,y) = 1 - 3x^2y^2+x^2y^4+x^4y^2$ is positive on $\mathbb{R}^2$, but it is not a sum of squares in $\mathbb{R}[x,y]$.
\end{proposition}
\begin{proof}
Positivity follows directly by plugging in $a = 1$, $b= x^2y^4$, and $c=x^4y^2$ into the inequality of the arithmetic and geometric means
\[
\frac{a+b+c}{3} \geq (abc)^{1/3} (\textrm{if } a,b,c \geq 0).
\]
To show that $s(x,y)$ is not a sum of squares, suppose to the contrary that $s=f_1^2+\dots f_n^2$ for some polynomials $f_i \in \mathbb{R}[x,y]$. If $d = \max \{\deg (f_i)| i \in \{1,\dots,n\} \}$, then $s$ must be of degree $2d$ ($\leq 2d$ is obvious, and for the other direction let $f_{id}$ be the homogeneous part of degree $d$ of $f_i$. The part of degree $2d$ of $f$ is $f_{1d}^2+\dots +f_{nd}^2$ out of which at least one is nonzero by definition of $d$, and $\geq 2d$ follows). Since $s$ is of degree $6$ any $f_i$ can have degree at most $3$ and must therefore be a linear combination of the monomials
\[
1,x,y,x^2,xy,y^2,x^3,x^2y,xy^2,y^3.
\]
If an $x^3$-term would appear in any $f_i$ then an $x^6$ term would appear in $s$ with a positive coefficient. There is no way to cancel that term, so there is no $x^3$-term in any $f_i$. By the same argument we can conclude that there are no terms $x^2$, $x$, $y^3$, $y^2$ and $y$ either and so $f_1,\dots,f_n$ are on the form
\[
f_i = a_i + b_ixy+c_ix^2y+d_ixy^2.
\]
We get a contradiction by 
\[
\sum_{i=1}^nb_i^2 = -3.
\]
\end{proof}
Despite the negative result one could hope for general theorems when we change the settings slightly. In 1900 Hilbert posed a famous set of problems, of which the 17th asked was whether any positive polynomial can be written as a sum of squares of rational functions. Artin proved the general case of this problem in 1927:
\begin{theorem}
For any $f \in \mathbb{R}[x_1,\dots,x_n]$ it is true that $f \geq 0$ on $\mathbb{R}^n$ if and only if $f$ is a sum of squares of rational functions.
\end{theorem}

We are now ready to state the first version of the Positivstellensatz, which has some similarities with Hilbert's 17th problem. The main ideas were due to Krivine in 1964 \cite{Krivine1964}, and Stengle in 1974 \cite{Stengle1974}.  We present a version from \cite{Marshall2008} that fits with our notation.

\begin{theorem}[Positivstellensatz]
Let $f,g_1,\dots,g_m \in \mathbb{R}[x_1,\dots,x_n]$, $K= \{x \in \mathbb{R}^n | g_i(x) \geq 0, i =1,\dots, m\}$ and 
\[
T = \{ \sum_{e \in \{0,1\}^m} \sigma_e g^e | \sigma_e \textrm{ is a sum of squares for all } e \in \{0,1\}^m \}.
\]
\begin{itemize}
\item[(a)] $f>0$ on $K$ if and only if there exists $p,q \in T$ such that $pf = 1+q$.
\item[(b)] $f \geq 0$ on $K$ if and only if there exists an integer $m \geq 0$ and $p,q \in T$ such that $pf = f^{2m} + q$.
\item[(c)] $f = 0$ on $K$ if and only if there exists an integer $m \geq 0$ such that $-f^{2m} \in T$.
\item[(d)] $K = \emptyset$ if and only if $-1 \in T$.
\end{itemize}
\end{theorem}
Recall that we want to find a good way of minimizing $f$. This would be possible if we could find the largest $\lambda$ such that (b) holds for $f-\lambda$. The problem is that there is no algorithmic way of doing this since we cannot control the degrees in this version of the Positivstellensatz. Although the theorem is very important theoretically it is not satisfactory for our application. To get a certificate for positivity that is more suitable we need to impose further conditions on $K$. In the case when $K$ is compact Scm\"udgen managed to prove this version of the Positivstellensatz in 1991 \cite{Schmudgen1991}:
\begin{theorem}[Schm\"udgen's Positivstellensatz]
Let $K = \{ x \in \mathbb{R}^n | g_1(x)\geq 0, \dots, g_m(x) \geq 0\}$ be compact. If $f$ is strictly positive on $K$, then 
\[
f= \sum_{e \in \{0,1\}^m} \sigma_e g^e
\]
 where $\sigma_e$ is a sum of squares for all $e \in \{0,1\}^m$.
\end{theorem}

This version is a major improvement, but in terms of computational efficiency it could be improved further as it is exponential in terms of the number of boundary conditions $g_i$. Two years later Putinar added an additional condition on the polynomials to make another version of the Positivstellensatz \cite{Putinar1993} that turns out to work well with semidefinite relaxations, and which is only linear in the number of boundary conditions $g_i$. Let us introduce the condition, state the theorem and then explore when the condition is satisfied.

\begin{definition}
Given polynomials $g_1,\dots,g_m \in \mathbb{R}[x_1,\dots,x_n]$, the \emph{quadratic module} generated by $g_1,\dots,g_m$ is defined by:
\[
\mathrm{QM}(g_1,\dots,g_m) = \{\sigma_0 + \sum_{i=1}^m \sigma_ig_i | \sigma_0,\dots,\sigma_m \textrm{ are sums of squares} \} . 
\]
\end{definition}

\begin{definition}
A quadratic module $\mathrm{QM}(g_1,\dots,g_m) $ is \emph{Archimedean} if
\[ N - \sum_{i=1}^n x_i^2 \in
\mathrm{QM}(g_1,\dots,g_m)
\]
for some $N \in \mathbb{N}$.
\end{definition}

\begin{theorem}[Putinar's Positivstellensatz]
\label{thm:Putinar}
Let 
\[
K = \{ x \in \mathbb{R}^n | g_1(x) \geq 0 , \dots , g_m(x) \geq 0\}
\]
be compact. If $f$ is strictly positive on $K$ and the associated quadratic module  $\mathrm{QM}(g_1,\dots,g_m) $ is Archimedean, then $f= \sigma_0 + \sum_{i=1}^m \sigma_i g_i$ where $\sigma_1,\dots,\sigma_m$ are sums of squares.
\end{theorem}
Putinar asked the question \emph{Assuming $K$ is compact, is it true that $\mathrm{QM}$ is Archimedean?} which would then allow us to remove this additional condition from the theorem.
It is true and easy to show when $m=1$, and is highly nontrivial but still true when $m=2$, which was proven by Jacobi and Prestel in 2001 \cite{Jacobi2001}. It is also true if the ring $\frac{\mathbb{R}[x_1,\dots,x_n]}{M \cap -M}$ has Krull dimension $\leq 1$, but in general it is not true if $m \geq 3$ and the Krull dimension is 2 or more. The different cases are discussed in further detail starting at page 97 of \ref{Marshall2008}.

For the general case we provide the following conditions due to Schm\"udgen in 1991 \cite{Schmudgen1991} that are equivalent to the Archimedean condition:
\begin{theorem}
\label{Thm:Schmudgen1991}
The following are equivalent:
\begin{itemize}
\item[(1)] $\mathrm{QM}(g_1,\dots,g_m)$ is Archimedean.
\item[(2)] There exist finitely many polynomials $t_1,\dots,t_m \in \mathrm{QM}(g_1,\dots,g_m)$ such that the set
\[
\{ x \in \mathbb{R}^n | t_1(x) \geq 0, \dots, t_m(x) \geq 0 \}
\]
is compact and $\prod_{i \in I} t_i \in \mathrm{QM}(g_1,\dots,g_m)$ for all $I \subset \{1, \dots, m\}$
\item[(3)] There exists a polynomial $p \in \mathrm{QM}(g_1,\dots,g_m)$ such that $\{x \in \mathbb{R}^n | p(x) \geq 0 \}$ is compact.
\end{itemize}
\end{theorem}

We have introduced enough state-of-the-art tools to find sum of squares based certificates to check positivity of polynomials. Recall that our goal is to solve polynomial optimization problems:
\[ 
\begin{array}{rll}
\inf  & \displaystyle  f(x) \\
\textnormal{subject to} & \displaystyle  g_1(x) \geq 0, \dots, g_m(x) \geq 0, \\
& \displaystyle  x \in \mathbb{R}^n.
\end{array}
\]
Let
\[ K = \{x \in \mathbb{R}^n : g_j(x) \geq 0, j = 1,\dots,m\}
\]
denote the feasible set to the polynomial optimization problem. Sometimes we impose the following technical assumption:
\begin{assumption}
\label{ass:arch}
Let us assume that there exists a polynomial $\sigma \in \mathbb{R}[x_1,\dots,x_n]$ on the form
\[
\sigma=\sigma_0+\sum_{j=1}^m \sigma_jg_j,
\]
where $\sigma_i$ is a sum of squares for $i=1,\dots,m$, such that the set $\{ x \in \mathbb{R}^n | \sigma(x) \geq 0\}$ is compact.
\end{assumption}
As the assumption is equivalent to part (3) of Theorem \ref{Thm:Schmudgen1991} it guarantees that quadratic module generated by polynomials $g_1,\dots,g_m$ is Archimedean, and thus that Putinar's Positivstellensatz can be applied. When assumption \ref{ass:arch} holds it is clear that the following problem has the same optimal value as our original polynomial optimization problem:
\[
\begin{array}{rl}
\sup & \displaystyle  \lambda \\
\textnormal{subject to} & \displaystyle  f(X)-\lambda = \sigma_0 + \sum_{i=1}^m \sigma_ig_i \\
& \displaystyle  \sigma_0,\dots,\sigma_m \textrm{ are sums of squares}.
\end{array}
\]
In particular, $\lambda$ is now the variable and $X$ a formal indeterminate. Next we will explore the duality between these problems, and discuss how we can relax them to find lower bounds for polynomials on $K$.

\section{The moment problem and duality}
\label{sec:Mom}
To fully understand how to solve polynomial optimization problems we need to introduce a duality theory, and to introduce a duality theory we first need to study the moment problem. The problem we are interested in is a special version of the \emph{Generalized Moment Problem} (GMP). In its general form, a lot of problems from applied mathematics, statistics, probability theory, economics, engineering, physics and operation research can be stated in terms of a moment problem. Although the GMP provides a nice theoretical framework and has great modeling power, it cannot in its full generality be solved numerically. We do not state the theorems in their most general form, but rather in our context of minimizing a polynomial over a closed basic semialgebraic set. We refer to \cite{Lasserre2010} for the reader who is interested in a more complete background.

In our context $f,g_1,\dots,g_m$ are polynomial from $\mathbb{R}^n$ to $\mathbb{R}$ and $K = \{g_1\geq 0,\dots,g_m \geq 0\}$ a compact closed basic semialgebraic set. This is a particular instance of the GMP (Theorem 1.1 page 6, \cite{Lasserre2010}).
\begin{theorem}
Let  $\mathscr{M}(K)_+$ be the space of finite Borel measures $\mu$ on $K$. The following two problems are equivalent in the sense that $f^* = \rho_{\textnormal{mom}}$.
\begin{equation} 
\label{eq:prim1}
\begin{array}{rll}
f^* = \displaystyle  \inf  & \displaystyle  f(x) \\
 \textnormal{subject to} & \displaystyle  x \in K.
\end{array}
\end{equation}
\begin{equation}
\label{eq:prim2}
\begin{array}{rl}
\rho_{\textnormal{mom}} = \displaystyle \inf_{\mu \in \mathscr{M}(K)_+} & \displaystyle \int_K f d\mu \\
 \textnormal{subject to} & \displaystyle \int_K d\mu=1.
\end{array}
\end{equation}
\end{theorem}
Since \eqref{eq:prim2} is a linear program we find its dual through the standard procedure in linear programming:
\begin{equation}
\label{eq:dual1}
\begin{array}{rl}
\rho_{\textnormal{pop}} = \displaystyle \sup & \lambda \\
 \textnormal{subject to} & \displaystyle  f(x) - \lambda \geq 0 \textrm{ for all } x \in K.
\end{array}
\end{equation}

In general strong duality holds when we Slater's condition is satisfied. In our case, when we are minimizing a polynomial over a compact feasible set $K$, we do not need Slater's condition for strong duality to hold (Special case of Theorem 1.3 page 8, \cite{Lasserre2010}):
\begin{theorem}
\label{thm:duality}
Suppose that $K$ is compact, that $f \in \mathbb{R}[x]$ and that we have the primal and dual pair \eqref{eq:prim2} and \eqref{eq:dual1}. Then $\rho_{mom} = \rho_{pop}$ and if \eqref{eq:prim2} has a feasible solution then it has an optimal solution.
\end{theorem}

Let $v_r(x) = [1,x_1,\dots,x_n,x_1x_2, \dots, x_n^2,\dots,x_1\cdots x_r, \dots,x_n^r]^T$ be the vector of all monomials up to degree $r$ (order first reversed graded lexicographically within groups with same set of exponentials, then the groups are ordered reversed graded lexicographically by leading terms in the groups).  Let similarly $v(x)$ be the infinite sequence of monomials $v(x) = v_{r \rightarrow \infty}(x)$. Let $y = (y_\alpha) \subset \mathbb{R}$ be an infinite sequence indexed by $\alpha \in \mathbb{N}^n$ ordered the same way as the monomials in $v_r$, and define $L_y : \mathbb{R}[x] \rightarrow \mathbb{R}$ by
\begin{equation}
\label{def:L_y}
f(x) = \displaystyle \sum_{\alpha \in \mathbb{N}^n} f_\alpha x^\alpha \mapsto L_y(f) = \sum_{\alpha \in \mathbb{N}^n} f_\alpha y_\alpha.
\end{equation}
Let ${\bf M}_r(y)$ be the \emph{moment matrix} defined by
\begin{equation}
\label{def:M_r}
{\bf M}_r(y)(\alpha,\beta) = L_y(x^\alpha x^\beta)=y_{\alpha+\beta}, \textrm{ for all } \alpha,\beta \in \mathbb{N}_r^n.
\end{equation}
Equivalently ${\bf M}_r(y) = L_y(v_r(x)v_r(x)')$, which is why we say that the rows and columns of ${\bf M}_r(y)$ are labeled by $v_r(x)$.

Given a polynomial $u \in \mathbb{R}[x]$ with coefficient vector ${\bf u} = \{ 
g \}$ we define the \emph{localizing matrix} to be the matrix ${\bf M}_r(uy)$, obtained from ${\bf M}_r(y)$ by:
\[ 
{\bf M}_r(uy)(\alpha, \beta) = L_y(u(x)x^\alpha x^\beta) = \displaystyle \sum_{g \in \mathbb{N}^n} u_g y_{g+\alpha+\beta}, \textrm{ for all } \alpha, \beta \in \mathbb{N}_r^n.
\]

To understand the notation let us consider two examples.
\begin{example}
When $n=2$, $r=1$ and $u(x) = a-x_1^2+bx_2^2$ we have 
\[v_r(x)=(1,x_1,x_2),\] 
\[
y=\big(y_{00},y_{10},y_{01},y_{11},y_{20},y_{02}, \dots \big), 
\]
\[
{\bf M}_1(y) = 
\left[\begin{array}{cccccc}
y_{00} & y_{10} & y_{01}  \\
y_{10} & y_{20} & y_{11}  \\
y_{01} & y_{11} & y_{02}
\end{array} \right],
\]
and
\[
{\bf M}_1(uy) = 
\left[\begin{array}{cccccc}
ay_{00}-y_{20}+by_{02} & ay_{10}-y_{30}+by_{12} & ay_{01}-y_{21}+by_{03}  \\
ay_{10}-y_{30}+by_{12} & ay_{20}-y_{40}+by_{22} & ay_{11}-y_{31}+by_{13}  \\
ay_{01}-y_{21}+by_{03} & ay_{11}-y_{31}+by_{13} & ay_{02}-y_{22}+by_{04}
\end{array} \right],
\]
\end{example}
\begin{example}
When $n=2$, $r=2$ and $u=ax_1^3$ we have 
\[v_r(x)=(1,x_1,x_2,x_1x_2,x_1^2,x_2^2),\] 
\[
y=\big(y_{00},y_{10},y_{01},y_{11},y_{20},y_{02},y_{21},y_{12},y_{30},y_{03},y_{22}, y_{31},y_{13},y_{40},y_{04}, \dots \big), 
\]
\[
{\bf M}_2(y) = 
\left[\begin{array}{cccccc}
y_{00} & y_{10} & y_{01} & y_{11} & y_{20} & y_{02} \\
y_{10} & y_{20} & y_{11} & y_{21} & y_{30} & y_{12} \\
y_{01} & y_{11} & y_{02} & y_{12} & y_{21} & y_{03} \\
y_{11} & y_{21} & y_{12} & y_{22} & y_{31} & y_{13} \\
y_{20} & y_{30} & y_{21} & y_{31} & y_{40} & y_{22} \\
y_{02} & y_{12} & y_{03} & y_{13} & y_{22} & y_{04} \\
\end{array} \right]
\]
and
\[
{\bf M}_2(uy) = 
\left[\begin{array}{cccccc}
ay_{30} & ay_{40} & ay_{31} & ay_{41} & ay_{50} & ay_{32} \\
ay_{40} & ay_{50} & ay_{41} & ay_{51} & ay_{60} & ay_{42} \\
ay_{31} & ay_{41} & ay_{32} & ay_{42} & ay_{51} & ay_{33} \\
ay_{41} & ay_{51} & ay_{42} & ay_{52} & ay_{61} & ay_{43} \\
ay_{50} & ay_{60} & ay_{51} & ay_{61} & ay_{70} & ay_{52} \\
ay_{32} & ay_{42} & ay_{33} & ay_{43} & ay_{52} & ay_{34} \\
\end{array} \right].
\]
\end{example}

The following theorem, which is due to Haviland and Riesz, is one of the main building blocks to many proofs in duality theory (Theorem 3.1 page 53, \cite{Lasserre2010}):
\begin{theorem}
\label{thm:Haviland}
Let $y=(y_\alpha)_{\alpha \in \mathbb{N}^n} \subset \mathbb{R}$ be an infinite sequence, and let $K \subset \mathbb{R}^n$ be closed. There exists a finite Borel measure $\mu$ on $K$ such that
\[
\displaystyle \int_K x^\alpha d\mu = y_\alpha, \textrm{ for all } \alpha \in \mathbb{N}^n
\]
if and only if $L_y(p) \geq 0$ for all polynomials $p \in \mathbb{R}[x]$ that are nonnegative on $K$.
\end{theorem}
One of the theorems based on Theorem  \ref{thm:Haviland} is the following  (Theorem 3.8 page 63, \cite{Lasserre2010}):
\begin{theorem}
\label{thm:dualpositivstellensatz}
Let $y = (y_\alpha)_{\alpha \in \mathbb{N}^n}$ be a given infinite sequence in $\mathbb{R}$, introduce the polynomials $g_1,\dots,g_m \in \mathbb{R}[x]$ and let $K \subset \mathbb{R}^n$ be the closed basic semi-algebraic set
\[ 
K = \{ x \in \mathbb{R}^n | g_i(x) \geq 0, i = 1,\dots, m \}.
\]
Assume that $K$ is compact.
\begin{itemize}
\item[(a)] There exists a finite Borel measure $\mu$ on $K$ such that $\displaystyle \int_K x^\alpha d\mu = y_\alpha$ for all $\alpha \in \mathbb{N}^n$ if and only if ${\bf M}_r(g_Jy) \succeq 0$ for all $J \subseteq \{1, \dots, m\}$ and all $r \in \mathbb{N}$. 
\item[(b)] Assume that Assumption \ref{ass:arch} holds, then there exists a finite Borel measure $\mu$ on $K$ such that $\displaystyle \int_K x^\alpha d\mu = y_\alpha$ for all $\alpha \in \mathbb{N}^n$ if and only if ${\bf M}_r(g_jy) \succeq 0$ and ${\bf M}_r(y) \succeq 0$ for all $j \in \{ 1, \dots, m\}$ and all $r \in \mathbb{N}$.
\end{itemize}
\end{theorem}

If we let $ y_\alpha = \displaystyle \int_K x^\alpha d\mu$ we get $\displaystyle \int_K f d\mu = \displaystyle \sum_{\alpha \in \mathbb{N}^n} f_\alpha y_\alpha = L_y(f)$ and $\displaystyle \int_K d\mu = y_0 + \sum_{\alpha \in \mathbb{N}^n \smallsetminus 0} 0y_\alpha = y_0$, and can restate \eqref{eq:prim2} as
\begin{equation}
\label{eq:prim3}
\begin{array}{rl}
\rho_{\textnormal{mom}} = \displaystyle  \inf_{y} & L_y(f), \\
\textnormal{subject to}  & \displaystyle  y_0 = 1, \\
&  y_\alpha = \displaystyle \int_K x^\alpha d\mu, \alpha \in \mathbb{N}^n, \textrm{ for some } \mu \in \mathscr{M}(K)_+.
\end{array}
\end{equation}
If we let $v_j = \deg(g_j)$ and apply Theorem \ref{thm:dualpositivstellensatz} we get 
\begin{equation}
\label{eq:prim4}
\begin{array}{rl}
\rho_{\textnormal{mom}} = \displaystyle  \inf_{y} & L_y(f), \\
\textnormal{subject to}  & \displaystyle  y_0 = 1, \\
& \displaystyle  {\bf M}_i(y) \succeq 0, \textrm{ for all } i \in \mathbb{N}, \\
& \displaystyle  {\bf M}_{i-v_j}(yg_j) \succeq 0,\forall i \geq v_j, j = 1,\dots, m.
\end{array}
\end{equation}

When assumption \ref{ass:arch} holds we get the dual problem through \eqref{eq:dual1} and Putinar's Positivstellensatz:
\begin{equation}
\label{eq:dual2}
\begin{array}{rl}
\rho_{\textnormal{pop}} = \displaystyle  \sup_{\lambda, \{ \sigma_i\} }  & \lambda \\
\textnormal{subject to} & \displaystyle  f - \lambda= \sigma_0 + \sum_{i=1}^m \sigma_ig_i \geq 0, \\
& \displaystyle  \sigma_i \textnormal{ is a sum of squares for } i = 0,\dots,m.
\end{array}
\end{equation}
That $\sigma_i$ is a sum of squares is equivalent to that $\sigma_i=v(x)^TQ_iv(x)$ for some positive semidefinite matrix $Q_i$, hence we get the equivalent formulation
\[ 
\begin{array}{rl}
\rho_{\textnormal{pop}} =  \displaystyle \sup_{\lambda,\{ Q_i \}}  & \lambda \\
\textnormal{subject to} &  \displaystyle  f(X) - \lambda = v(X)^TQ_0v(X) + \sum_{i=1}^m v(X)^TQ_iv(X)g_i(X), \\
& \displaystyle \lambda \in \mathbb{R} \\
& \displaystyle  Q_i \succeq 0 \textnormal{ for } i = 1,\dots,m.
\end{array}
\]
Since the equality has to hold for every monomial separately we can rewrite the problem in order to get the problem on the form of a semidefinite program. If we use the notation $h(X)= \sum_\alpha [h(X)]_\alpha X^\alpha$, we get the equivalent problem 
\begin{equation}
\label{eq:dual3}
\begin{array}{rl}
\rho_{\textnormal{pop}} = \displaystyle \sup_{\lambda, \{Q_i\}}  & \lambda \\
\textnormal{subject to} & [\displaystyle f(X) - v(X)^TQ_0v(X) + \sum_{i=1}^m v(X)^TQ_iv(X)g_i(X)]_0 = \lambda, \\
& \displaystyle [f(X) - v(X)^TQ_0v(X) + \sum_{i=1}^m v(X)^TQ_iv(X)g_i(X)]_\alpha = 0, \\
& \displaystyle  \lambda \in \mathbb{R}, \\
& \displaystyle  Q_i \succeq 0 \textnormal{ for } i = 1,\dots,m, \\
\end{array}
\end{equation}
where $[f(X) - v(X)^TQ_0v(X) + \sum_{i=1}^m v(X)^TQ_iv(X)g_i(X)]_g$ is a linear polynomial in the entries of the matrices $Q_i$. Since the objective function is linear and constraints either linear or linear matrix inequalities this is indeed a semidefinite program. 

The primal and dual problems \eqref{eq:prim4} and \eqref{eq:dual2} have the same optimal values as the original pair \eqref{eq:prim2} and \eqref{eq:dual1}, hence strong duality holds by Theorem \ref{thm:duality}. As you might have noticed we have worked our way to some infinite-dimensional semidefinite programs both in the primal and dual. What the primal and dual problems also share is that they are easy to relax. In order to find a lower bound to the primal problem \eqref{eq:prim4} we only use moment matrices and localization matrices of degree $r$, and to relax the dual problem \eqref{eq:dual3} we can limit the degrees by replacing $v(x)$ with $v_r(x)$. Let $s$ be a positive integer, we have the following hierarchy of primal and dual relaxations, commonly referred to as the Lasserre hierarchy:
\begin{equation}
\label{eq:prim5}
\begin{array}{rl}
\rho_s^* = \displaystyle  \inf_{y} & L_y(f), \\
\textnormal{subject to}  & \displaystyle  y_0 = 1, \\
& \displaystyle  {\bf M}_{\lfloor s/2 \rfloor} (y) \succeq 0, \\
& \displaystyle  {\bf M}_{\lfloor (s-\deg(g_j))/2 \rfloor}(yg_j) \succeq 0, \forall j \in \{1,\dots, m\}, 
\end{array}
\end{equation}
and
\begin{equation}
\label{eq:dual4}
\begin{array}{rl}
\rho_s =\displaystyle \sup_{\lambda, \{ \sigma_i\} }  & \lambda \\
\textnormal{subject to} & \displaystyle  f - \lambda = \sigma_0 + \sum_{i=1}^m \sigma_ig_i \geq 0, \\
& \displaystyle  \sigma_0 \textnormal{ is a sum of squares of order } 2\lfloor s/2 \rfloor, \\
& \displaystyle  \sigma_i \textnormal{ is a sum of squares of order } 2 \lfloor (s-\deg(g_i))/2\rfloor.
\end{array}
\end{equation}
The following important convergence result by Lasserre (similar to Theorem 3.4 page 805 \cite{Lasserre2001}, and Theorem 4.1 page 79 \cite{Lasserre2010}) holds:
\begin{theorem}
Let $f,g_1,\dots,g_m \in \mathbb{R}[x]$, let $K= \{g_1 \geq 0,\dots,g_m \geq 0 \}$ be compact, and assume assumption \ref{ass:arch} holds. Let $\rho_{\textnormal{mom}}$ be the optimal value of \eqref{eq:prim4}, assumed to be finite, let $\rho_{\textnormal{pop}}$ be the optimal value of \eqref{eq:dual2}, and consider the sequence of primal and dual semidefinite relaxations defined in \eqref{eq:prim5} and \eqref{eq:dual4} with optimal values $(\rho_i)$ and $(\rho_i^*)$ respectively. Then $\rho_i^* \rightarrow \rho_{\textnormal{mom}}$ and $\rho_i \rightarrow \rho_{\textnormal{mom}}$ when $i \rightarrow \infty$.
\end{theorem}
\begin{proof}
By Theorem \ref{thm:duality} we know that $\rho_{\textnormal{pop}}=\rho_{\textnormal{mom}}$. Furthermore we know that $\rho_{\textnormal{mom}} \geq \rho_i^*$ since any feasible solution to the primal is a feasible solution to its relaxation, $\rho_{\textnormal{pop}} \geq \rho_i$ since every solution to the dual relaxation is a solution to the dual, and $\rho_i \leq \rho_i^*$ because of weak duality in semidefinite programming. Thus we have
\[ 
\rho_i \leq \rho_i^* \leq \rho_{\textnormal{pop}}= \rho_{\textnormal{mom}}, \textnormal{ for all } i\geq \displaystyle \max_{j}(\deg(g_j)).
\]
Let $\epsilon > 0$ and let $\lambda$ be a feasible solution to \eqref{eq:dual2} satisfying
\begin{equation}
\label{eq:chainofeq}
\rho_{\textnormal{pop}}-\epsilon=\rho_{\textnormal{mom}}-\epsilon \leq \lambda \leq   \rho_{\textnormal{pop}}=\rho_{\textnormal{mom}}.
\end{equation}
Since $\lambda$ is a feasible solution
\[
f - \lambda \geq 0, \textrm{ on } K.
\]
Let $\bar{\lambda} = \lambda - \epsilon$ be a new solution, it is indeed feasible since $f -\bar{\lambda} \geq \epsilon > 0$ on $K$. Using Putinar's Positivstellensatz we get
\[ 
f -\bar{\lambda} = \sigma_0 + \displaystyle \sum_{j=1}^m \sigma_jg_j,
\]
where $\sigma_0,\dots,\sigma_m$ are sums of squares. This implies that $\bar{\lambda}$ is a feasible solution to the relaxation \eqref{eq:dual4} if $2i\geq \displaystyle \max_j \sigma_jg_j$, which in turn implies that
\[ 
\rho_{\textrm{mom}} - 2\epsilon \leq \bar{\lambda} \leq\rho_i,
\]
where we got the first equality by removing $\epsilon$ from both sides of the first inequality in \eqref{eq:chainofeq}. Since we have
\[
\rho_{\textnormal{mom}}-2\epsilon \leq \rho_i \leq \rho_i^* \leq \rho_{\textnormal{mom}}
\]
and since we picked $\epsilon>0$ arbitrary the result follows.
\end{proof}

It is possible to check whether the sequence of semidefinite relaxations has converged at relaxation $i$ \cite{Henrion2005}, but there is no guarantee that it converges in a finite number of steps. For many practical purposes the semidefinite program becomes huge and not possible to solve for computational reasons already for quite small values of $i$. Even when there is no hope for an optimal solution to the original problem the relaxation provides a lower bound, and the higher the order of the relaxation the better the lower bound.

\section{Exploiting symmetries in semidefinite programming}
\label{sec:Sym}
In this section we explain how one can exploit symmetries in semidefinite programming. These methods or similar methods have been used in one way or another in the applications that are introduced in this section. For more details we refer to the references in each of the applications.

Let  $C$ and $A_1, \dots, A_m$ be real symmetric matrices and $b_1,\dots,b_m$ real numbers. In this section we provide the tools to reduce the order of the matrices in the semidefinite programming problem
\[ \max \{\mathrm{tr}(CX) ~ | ~X \textrm{ positive semidefinite},  \mathrm{tr}(A_i X)=b_i \textrm{ for } i = 1, \dots, m\} \]
when it is invariant under a group acting on its variables.

Inspired by \cite{KPSchrijver2007} and \cite{Klerk2011}, we are using a $\ast$--representation in order to reduce the dimension of the problem. For an introduction to $\ast$--representations we refer to the book by Takesaki \cite{Takesaki2002}. This method as well as other state-of-the art methods for invariant semidefinite programs are discussed in \cite{Bachoc2012}. Other important recent contributions include \cite{Kanno2001, Gatermann2004, Vallentin2009, Maehara2010, Murota2010, Riener2013}.

\begin{definition}
A \emph{matrix $\ast$-algebra} is a collection of matrices closed under addition, scalar and matrix multiplication, and transposition.
\end{definition}

Let $G$ be a finite group acting on a finite set $Z$ and define a homomorphism $h : G \rightarrow S_{|Z|}$ where $S_{|Z|}$ is the group of all permutations of $Z$. For every element $g \in G$ we have a permutation $h_g=h(g)$ of $Z$, for which  $h_{g g'}=h_{g}h_{g'}$ and $h_{g^{-1}}=h_{g}^{-1}$. For every permutation $h_{g}$, define the corresponding permutation matrix $M_{g} \in \{0,1\}^{|Z| \times |Z|}$ element-wise by
\[
(M_{g})_{i,j}=
\left\{
\begin{array}{rl}
 1 & \textrm{ if } h_{g}(i )= j, \\ 
 0 & \textrm{otherwise} 
\end{array}
\right.
\]
for all $i,j \in Z$.
Let the span of these matrices define the matrix $\ast$-algebra
\[
\mathcal{A} = \left\{ \sum_{g \in G} \lambda_g M_{g} ~ | ~ \lambda_g \in \mathbb{R} \right\}.
\]
The matrices $X$ satisfying $XM_{g}=M_{g}X$ for all $g \in G$ are the \emph{invariant matrices} of $G$. The collection of all such matrices, 
\[
\mathcal{A'} = \{X \in \mathbb{R}^{n \times n} | XM=MX \textrm{ for all } M \in \mathcal{A} \},
\]
is the \emph{commutant} of $\mathcal{A}$, which again is a $\ast$-algebra. Let $d=\dim\mathcal{A'}$ be the dimension of the commutant.

$\mathcal{A'}$ has a basis of $\{0,1\}$-matrices $E_1,\dots,E_d$ such that $\sum_{i=1}^d E_i = J$, where $J$ is the matrix of size $|Z| \times |Z|$ with all ones.

For every $i=1,\dots,d$ we normalize $E_i$ to 
\[
B_i = \frac{1}{\sqrt{tr(E_i^TE_i)}}E_i, 
\]
so that  $\mathrm{tr}(B_i^TB_j) = \delta_{i,j}$ where $\delta_{i,j}$ is the Kronecker delta.

We define the \emph{multiplication parameters} $\lambda_{i,j}^k$ by 
\[B_iB_j = \sum_{k=1}^d \lambda_{i,j}^kB_k \]
for $i,j,k = 1,\dots,d$.

The multiplication parameters are used to define $d\times d$-matrices $L_1,\dots,L_d$ by
\[
(L_k)_{i,j} = \lambda_{k,j}^i 
\]
for $k,i,j =1,\dots,d$. The matrices $L_1,\dots,L_d$ spans the linear space
\[ 
\mathcal{L}=\{\sum_{i=1}^d x_iL_i : x_1,\dots,x_d \in \mathbb{R} \} 
\]

\begin{theorem}[\cite{KPSchrijver2007}]
The linear function $\phi : \mathcal{A'} \rightarrow \mathbb{R}^{d \times d}$ defined by $\phi(B_i) = L_i$ for $i=1,\dots,d$ is a bijection. Additionally, the linear function also satisfies $\phi(XY)=\phi(X)\phi(Y)$ and $\phi(X^T)=\phi(X)^T$ for all $X,Y \in \mathcal{A}'$.
\end{theorem}

\begin{corollary}[\cite{KPSchrijver2007}]
$\sum_{i=1}^d x_iB_i$ is positive semidefinite if and only if $\sum_{i=1}^d x_iL_i$ is positive semidefinite.
\label{cor:schrijver}
\end{corollary}

Given that it is possible to find a solution $X \in \mathcal{A}'$ we can use Corollary \ref{cor:schrijver} to reduce the order of the semidefinite constraint.

\begin{lemma}
\label{lem:groupaverage}
There is a solution $X \in \mathcal{A}'$ to a $G$-invariant semidefinite program
\[ 
\max \{\mathrm{tr}(CX) ~ | ~X \textrm{ positive semidefinite},  \mathrm{tr}(A_i X)=b_i \textrm{ for } i = 1, \dots, m\}. 
\]
\end{lemma}
\begin{proof}
Let $C,A_1,\dots,A_m$ be $|Z| \times |Z|$ matrices commuting with $M_g$ for all $g \in G$. If $X$ is an optimal solution to the optimization problem then the group average, $X'=\frac{1}{|G|} \sum_{g \in G}M_g X M_g^T$, is also an optimal solution:
It is feasible since 
\[
\begin{array}{rl}
\displaystyle \mathrm{tr}(A_jX') &= \displaystyle  \mathrm{tr}(A_j\frac{1}{|G|} \sum_{g \in G}M_g X M_g^T) \\
&= \displaystyle  \mathrm{tr}(\frac{1}{|G|} \sum_{g \in G}M_g A_jX M_g^T) \\
&= \displaystyle  \mathrm{tr}(\frac{1}{|G|} \sum_{g \in G}A_jX ) \\
&= \displaystyle  \mathrm{tr}(A_jX),
\end{array}
\]
where we have used that the well-known fact that the trace is invariant under change of basis. By the same argument  $\mathrm{tr}(CX') =\mathrm{tr}(CX)$, which implies that $X'$ is optimal. It is easy to see that $X' \in \mathcal{A}'$.
\end{proof}

All in all we get the following theorem:

\begin{theorem}[\cite{KPSchrijver2007}]
\label{thm:Schrijver}
The $G$-invariant semidefinite program
\[
\max \{\mathrm{tr}(CX) ~ | ~X \succeq 0,  \mathrm{tr}(A_i X)=b_i \textrm{ for } i = 1, \dots, m\} 
\]
has a solution $X = \sum_{i=1}^d x_iB_i$ that can be obtained by
\[
\max \{\mathrm{tr}(CX) ~ | ~ \sum_{i=1}^d x_iL_i \succeq 0, \mathrm{tr}(A_i \sum_{j=1}^d B_jx_j)=b_i \textrm{ for } i = 1, \dots, m\}.
\]
\end{theorem}

Dimension reduction of invariant semidefinite programs is useful in a wide range of combinatorial problems including Lov\'asz $\vartheta$ number, crossing numbers, error-correcting codes and kissing numbers. One can also do dimension reduction in quadratic assignment problems, which for example has been used to find a new relaxation of the traveling salesman problem. We briefly explain these applications below, and a more extensive survey of these and other recent applications can be found in \cite{Klerk2010, Bachoc2012}. The methods are also useful to reduce the dimension of the invariant semidefinite programs arising when counting arithmetic progressions, which will be seen in later chapters. Many of the results in this thesis relies on computations that would not be possible without this step.

\subsection{Lov\'asz $\vartheta$-number}
The $\vartheta$ number of a graph $G=(V,E)$ is the optimal value of the semidefinite program
\[
\begin{array}{rl}
\vartheta(G)= \displaystyle \sup & \mathrm{tr}(JX), \\
 \textnormal{subject to} & X_{ij} = 0 \textrm{ for all }\{i,j\} \in E (i \neq j), \\
& \mathrm{tr}(X) = 1, \\
& X \succeq 0.
\end{array}
\]
One of the main reasons to study $\vartheta(G)$ is that it is "sandwiched" between the independence number $\alpha(G)$, the size of the largest clique in $G$, and $\chi(\bar{G})$, the chromatic number of the complement of $G$. We have the following theorem from Lov\'asz in 1979 \cite{Lovasz1979}:
\begin{theorem}[Lov\'asz sandwich theorem]
\[
\alpha(G) \leq \vartheta(G) \leq \chi(\bar{G}).
\]
\end{theorem}
To compute $\alpha(G)$ and $\chi(G)$ are NP-complete problems, thus if one can find $\vartheta$ quickly it gives valuable bounds for the other two graph properties when they cannot be obtained. For a long exposition of the $\vartheta$ number we refer to the notes by Knuth on the subject \cite{Knuth1993}.

When there are a lot of symmetries in the graph one can use the described techniques to reduce the dimension of the semidefinite program. This allows calculation of the $\vartheta$-number for larger graphs than otherwise possible \cite{Dukanovic2007, Klerk2009}.

\subsection{Block codes} 
Let us fix an alphabet ${\bf q} = \{0,1,\dots,q-1\}$ for some integer $q \geq 2$.  ${\bf q}^n$ is called the  \emph{Hamming space} and it is equipped with a metric $d(\cdot,\cdot)$ called \emph{Hamming distance}, which is given by
\[
d(u,v) = | \{ i : u_i \neq v_i \} | \textrm{ for all } u,v \in {\bf q}^n.
\]
A subset $C \subseteq {\bf q}^n$ is called a \emph{code} of length $n$, and for a nonempty code we define the \emph{minimum distance} of $C$ to be 
\[
\min \{ d(u,v) : u,v \in C \textrm{ are distinct} \} .
\]
With $A_q(n,d)$ we denote the maximum size of a code of length $n$ and minimum distance at least $d$: 
\[
A_q(n,d) = \max \{ |C| : C \in {\bf q}^n \textrm{ has minimum distance at least } d\}.
\]
$A_q(n,d)$ is an important quantity in coding theory, and it is of interest to find good upper and lower bounds. Lower bounds are usually obtained by finding explicit constructions of codes, whereas upper bounds are usually obtained using other methods.

The \emph{Hamming graph} $G_{\bf q}(n,d)$ is constructed on the vertex set ${\bf q}^n$ by connecting two words $u,v \in {\bf q}^n$ if $d(u,v) < d$. Since the codes of minimum distance at most $d$ correspond to independent sets in $G_q(n,d)$ we have
\[
A_q(n,d) = \alpha(G_q(n,d))
\]
where $\alpha$ is the independence number of a graph. Recall that Lovasz $\vartheta$ number is a bound for the independence number, and hence the previous subsection gives a first class of bounds.

It is also fairly easy to see that for the graph parameter 
\[
\vartheta'(G) = \max \{ \langle X, J \rangle : \langle X, I \rangle = 1, X_{uv}=0 \textrm{ if } uv \in E(G), X \geq 0, X \succeq 0 \},
\]
where $I$ is the identity and $J$ the all one matrix, we have $\alpha(G) \leq \vartheta'(G) \leq \vartheta(G)$. Thus one can try to find an upper bound $A_q(n,d)$ by solving the semidefinite program $\vartheta'(G_q(n,d))$. The semidefinite program has an exponential size, but by symmetry reduction it can be solved more efficiently. It was found independently in \cite{McEliece1978} and \cite{Schrijver1981} that the semidefinite program can be reduced to a linear program that gives the Delsarte bounds \cite{Delsarte1973}. Invariant semidefinite programming methods have been used to find sharper bounds in several papers including \cite{Schrijver2005, Gijswijt2006, Laurent2007, Laurent2009}.

\subsection{Crossing numbers}
The \emph{crossing number} $\mathrm{cr}(G)$ of a graph $G$ is the minimum number of pairwise intersection of edges when $G$ is drawn in a plane. Exact crossing numbers are only known in few examples, and to find exact crossing numbers of families of graphs is still an active area of research.

Zarankiewicz claimed in 1954 \cite{Zarankiewicz1954} that the crossing number $\mathrm{cr}(K_{m,n})$ of the complete bipartite graph equals the Zarankiewicz number, which is given by $Z(m,n)=\lfloor \frac{m-1}{2} \rfloor \lfloor \frac{m}{2} \rfloor \lfloor \frac{n-1}{2} \rfloor \lfloor \frac{n}{2} \rfloor$. Ringel and Kainen independently found a gap in the proof and Zarankiewicz's claim has been conjectured since. Zarankiewicz argument for the upper bound was correct, and thus it is known that
\[ 
\mathrm{cr}(K_{m,n}) \leq Z(m,n)
\]
for any positive integers $m,n$, and so one usually tries to improve on the lower bound. It is also conjectured that the crossing number of the complete graph $\mathrm{cr}(K_n)$ equals $Z(n) = \lfloor \frac{n}{2} \rfloor \lfloor \frac{n-1}{2} \rfloor \lfloor \frac{n-2}{2} \rfloor \lfloor \frac{n-3}{2} \rfloor$.
Using semidefinite programming it was found in \cite{Klerk2006} that for each $m \geq 9$ it holds that 
\[
\begin{array}{l}
\lim_{n \rightarrow \infty} \mathrm{cr}(K_{m,n}) / Z(m,n) \geq 0.83m/(m-1), \\
\lim_{n \rightarrow \infty} \mathrm{cr}(K_{n,n}) / Z(n,n) \geq 0.83, \\
\lim_{n \rightarrow \infty} \mathrm{cr}(K_{n}) / Z(n) \geq 0.83.
\end{array}
\]
By carefully reducing the semidefinite program by exploiting symmetries the bound for $\mathrm{cr}(K_{m,n})$ was improved \cite{KPSchrijver2007} to
\[
\begin{array}{l}
\lim_{n \rightarrow \infty} \mathrm{cr}(K_{m,n}) / Z(m,n) \geq 0.8594m/(m-1), \\
\end{array}
\]
for all $m \geq 9$. The computations heavily relies on the symmetry reduction; without them one would need to solve a semidefinite program with $40320 \times 40320$ matrices, which is still far from possible with current SDP-solvers. In \cite{Klerk2011} the matrices in the semidefinite program were further block-diagonalized, which significantly improved computation time.

\subsection{Kissing numbers} 
The \emph{kissing number} is the maximum number $\tau_n$ of unit spheres with no pairwise overlap that can simultaneously touch the unit sphere in $n$-dimensional Euclidean space.

For the first non-trivial dimension, $n=3$, Sch\"utte and van der Waerden \cite{Schutte1952} found that $\tau_3 = 12$. It is also known that $\tau_4=24$ \cite{Musin2008}, and that $\tau_8=240$ and $\tau_{24}=196560$ \cite{Odlyzko1979, Levenshtein1979}. For other dimensions only bounds are known. A way to find upper bounds using linear programming was developed by Delsarte \cite{Delsarte1973}, and Delsarte, Goethals and Seidel \cite{Delsarte1977}. These bounds were improved by Bachoc and Vallentin \cite{Bachoc2008}, and Mittelmann and Vallentin \cite{Mittelmann2010} for all $n \leq 24$ using a semidefinite program that strengthens the linear program. The symmetry of the sphere is essential to obtain a finite SDP relaxation.

\subsection{Quadratic assignment problems}
A \emph{quadratic assignment problem} is a problem on the form 
\[
\min_{X \in \Pi_n} \mathrm{tr}(AX^TBX)
\]
where $\Pi_n$ is the set of $n \times n$ permutation matrices, and $A$ and $B$ are symmetric $n \times n$ matrices. Solving the quadratic assignment problem is computationally difficult, it is known to be an NP-hard problem, and thus trying to find bounds using relaxations is the best one can do when $n$ is not very small.

From the quadratic assignment problem we can form the following SDP relaxation \cite{Zhao1998, Klerk2010_2}:
\[
\begin{array}{rl}
 \min & \mathrm{tr}(A \otimes B) Y \\
 \textrm{subject to} & \mathrm{tr}((I \otimes (J-I)) Y + ((J-I) \otimes I ) Y) = 0 \\
 & \mathrm{tr}(Y) - 2 e^Ty = -n \\
&  \left( \begin{array}{cc} 1 & y^T \\ y & Y \end{array} \right) \succeq 0 \\
 & Y \geq 0
\end{array},
\]
where $\otimes$ is the Kronecker product, $I$ the identity matrix and $J$ the all one matrix. It is a relaxation since $Y = \mathrm{vec}(X)\mathrm{vec}(X)^T$ and $y = \mathrm{diag}(Y)$ is a feasible solution if $X \in \Pi_n$. Depending on the structure of $A$ and $B$ the sizes of the semidefinite constraint can sometimes be reduced significantly by the methods developed in this section. In \cite{Klerk2010_2, Klerk2012} the methods are discussed in the context of the quadratic assignment problem, and a lot of specific instances of the problem are discussed and solved. Computation times are improved in many instances and many bounds, especially for larger problems, are improved. 

The traveling salesman problem is the problem of finding a Hamiltonian circuit of minimum length in a graph, and it is easy to show that it is a special case of the quadratic assignment problem. Using the SDP relaxation and symmetry reduction of the quadratic assignment problem one gets the current strongest SDP relaxation of the traveling salesman problem \cite{Klerk2008}.

\subsection{Counting monochromatic arithmetic progressions in a 2-coloring}
One of the most difficult problems in Ramsey theory is to determine the existence of monochromatic arithmetic progressions in a group when its elements have been colored. The surveyed tools from real algebraic geometry makes it possible to count the progressions, a generalization of the existence problem, by stating it as a semidefinite program. The enumeration is done for the cyclic group in \cite{Sjoland_cyclic} and for any finite group in \cite{Sjoland_group}.

An arithmetic progression in a finite group $G$ of length $k \in \mathbb{Z}^+$ is a set of $k$ distinct elements, $\{a, b \cdot a, b\cdot b \cdot a,\dots, b^{k-1} \cdot a \}$, for any $a \in G$ and $b \in G \smallsetminus \{0\}$. In other words $\{1,2,3\}$, $\{1,3,2\}$, $\{2,1,3\}$, $\{2,3,1\}$, $\{3,1,2\}$ and $\{3,2,1\}$ should be considered as the same arithmetic progression. Thus when we sum over all arithmetic progressions only one representative for every arithmetic progression is used.

Let $\chi : G \rightarrow \{-1,1\}$ be a $2$-coloring of the group $G$, and let $x_g = \chi(g)$ for all $g \in G$. Furthermore, let $x$ be the vector of all variables $x_g$. For $a,b,c \in G$, let
\[
\begin{array}{rl}
p(x_a,x_b,x_c) &= \displaystyle \frac{(x_a+1)(x_b+1)(x_c+1)-(x_a-1)(x_b-1)(x_c-1)}{8}  \\
&= \displaystyle \frac{x_ax_b+x_ax_c+x_bx_c+1}{4} \\
&= \left\{ 
\begin{array}{ll}
1 & \text{if }  x_a=x_b=x_c\\
0 & \text{otherwise.}
\end{array} \right. 
\end{array}
\]
Hence $p$ is one if $a,b$ and $c$ are of the same color and zero otherwise. Let $R(3,G,2)$ denote the minimum number of monochromatic arithmetic progression of length 3 in a 2-coloring of the finite group $G$. It holds that:
\[
R(3,G,2)=  \min_{x \in \{-1,1\}^{n}} \displaystyle\sum_{\{a,b,c\} \textrm{ is an A.P. in } G}p(x_a,x_b,x_c).
\]
The problem can easily be relaxed to a semidefinite program with a lot of symmetries that can be exploited by the methods in this section. Using these techniques one can obtain the following results:

\begin{theorem}[\cite{Sjoland_cyclic}]
\label{thm:cycliccase}
Let $n$ be a positive integer and let $R(3,\mathbb{Z}_n,2)$ denote the minimal number of monochromatic $3$-term arithmetic progressions in any two-coloring of $\mathbb{Z}_n$. $n^2/8-c_1n+c_2 \leq R(3,\mathbb{Z}_n,2) \leq n^2/8-c_1n+c_3$ for all values of $n$, where the constants depends on the modular arithmetic and are tabulated in the following table.
\[
\begin{array}{c|c|c|c}
n \mod 24 & c_1 & c_2 & c_3 \\
\hline
1,5,7,11,13,17,19,23 & 1/2 & 3/8 & 3/8 \\
8,16 & 1 & 0 & 0 \\
2,10 & 1 & 3/2 & 3/2 \\
4,20 & 1 & 0 & 2 \\
14,22  & 1 & 3/2 & 3/2 \\
3,9,15,21 &  7/6 & 3/8 & 27/8 \\ 
0 &  5/3 & 0  & 0 \\ 
12 &  5/3 & 0 & 18 \\ 
6,18 &  5/3 &1/2 & 27/2 \\ 
\end{array}
\]
\end{theorem}

\begin{theorem}[\cite{Sjoland_group}]
\label{thm:groupcase}
Let $G$ be any finite group and let $R(3,G,2)$ denote the minimal number of monochromatic $3$-term arithmetic progressions in any two-coloring of $G$. Let $G_k$ denote the set of elements of $G$ of order $k$, $N=|G|$ and $N_k=|G_k|$. Denote the Euler phi function $\phi(k)=|\{ t \in \{1,\dots,k\}: t  \textrm{ and } k \textrm{ are coprime}\}|$. Let $K=\{k \in \{5,\dots,n\} : \phi(k) \geq \frac{3k}{4}\}$. For any $G$ there are $\sum_{k=4}^n \frac{N\cdot N_k}{2} + \frac{N \cdot N_3}{24}$ arithmetic progressions of length 3. At least
\[
\begin{array}{rl}
R(3,G,2) \geq &\displaystyle \sum_{k \in K}  \frac{N\cdot N_k}{8}(1- 3\frac{k-\phi(k)}{\phi(k)})
\end{array}
\]
of them are monochromatic in a 2-coloring of $G$.
\end{theorem}

\subsection{Counting arithmetic progressions in a fixed density set}
The famous theorem by Szemer\'edi states that there exist arithmetic progressions of any length in any set of the integers of positive density. After Szemer\'edi's result there has been a lot of research on how to improve the bounds on the number $N(\delta,k)$ for which any subset of density at least $\delta$ of the cyclic group $\mathbb{Z}_{N(\delta,k)}$ contains an arithmetic progressions of length $k$. Using real algebraic geometry one can study the more general question of counting the number of arithmetic progressions of length $k$ in any subset of $\mathbb{Z}_N$ of fixed density $N \delta$, denoted $W(3,\mathbb{Z}_N,\delta)$. It is easy to see that if $D=\delta N$ we have
\[ 
W(3,\mathbb{Z}_N,D/N) = \min \{\sum_{\{i,j,k\} \textrm{ is an A.P. in } \mathbb{Z}_N} x_ix_jx_k : x_i \in \{0,1\} , \sum_{i=0}^{N-1} x_i = D\},
\]
which can be relaxed and formulated as an invariant semidefinite program. The best lower bounds for $W(3,\mathbb{Z}_N,D/N)$ as well as a discussion on how to obtain a generalization of Szemer\'edi's theorem from this family of semidefinite programs can be found in \cite{Sjoland_density}. For example the following holds:
\begin{theorem}[\cite{Sjoland_density}]
 \label{thm:density}
Let $p$ be a prime number. A lower bound for the minimum number of arithmetic progressions of length $3$ among all subsets of $\mathbb{Z}_p$ of cardinality $D$,
\[  W(3,\mathbb{Z}_p,D/p) = \min \{ \sum_{\{i,j,k\} \textrm{ A.P. in } \mathbb{Z}_p} x_ix_jx_k : x_i \in \{0,1\} , \sum_{i=0}^{p-1} x_i = D\}, \]
is  
\[
\lambda = \frac{D^3 - (\frac{p+3}{2}) D^2 + (\frac{p+3}{2}-1)D}{p-1}.
\]
A certificate for the lower bound is given by:
\[ 
\begin{array}{rl}
\displaystyle \sum_{\{i,j,k\} \textrm{ A.P. in } \mathbb{Z}_p} X_iX_jX_k - \lambda =& \displaystyle  \sum_{i=0}^{p-1} \sigma_{1,i} X_i  + \sum_{i=0}^{p-1} \sigma_{2,i} X_i + \sigma_{3}(D- \sum_{i=0}^{p-1}X_i^3 ) \\
& \displaystyle  + \sigma_{4}(\sum_{i \neq j}X_i^2X_j - D(D-1)),
\end{array}
\]
where
\[ 
\begin{array}{rl}
\sigma_{1,i} &= \displaystyle  \frac{1}{p-1} \sum_{0<j<k<(p-1)/2} (X_{j+i} - X_{j+k+i} - X_{n-j-k+i} + X_{n-j+i})^2 \\ \\
\sigma_{2,i} &= \displaystyle  \frac{1}{p-1}(DX_i - \sum_{j=0}^{p-1} X_j)^2 \\ \\
\sigma_{3} &= \displaystyle \frac{(D-1)^2}{p-1} \\ \\
\sigma_{4} &= \displaystyle \frac{4D-p+3}{2(p-1)}.
\end{array}
\]
\end{theorem}

\section{Implementation}
The procedure is similar for any polynomial optimization problem, so for simplicity let $p$ be prime and let us consider the problem of counting the number of arithmetic progressions in a fixed density set:
\[
\min \{ \sum_{\{i,j,k\} \textrm{ A.P. in } \mathbb{Z}_p} x_ix_jx_k : x_i \in \{0,1\} , \sum_{i=0}^{p-1} x_i = D\},
\]
which is invariant under affine transformations; if $\{i,j,k\}$ is an arithmetic progression and $(a,b) \in \mathbb{Z} \rtimes \mathbb{Z}^*$, then also $\{a+bi,a+bj,a+bk\}$ is an arithmetic progression. We know from representation theory that we can split all possible $3$-sets into orbits;
\[
\bigcup_{0\leq i<j<k<n}\{i,j,k \} = S_1 \cup \dots  \cup S_D.
\]
In other words, if and only if both $\{i_1,j_1,k_1\} $ and $\{i_2,j_2,k_2\}$ are in $S_t$ it holds that $\{i_1,j_1,k_1\} = \{a+bi_2,a+bj_2,a+bk_2\}$ for some $(a,b)  \in \mathbb{Z} \rtimes \mathbb{Z}^*$.

  The optimization problem can be bounded from below using a degree 3 relaxation. Relaxing the integer constraints we have $0 \leq x_i \leq 1$, and the constraint $\sum_{i=0}^{p-1} x_i = D$ implies that $\sum_{i=0}^{p-1} x_i^2 = \sum_{i=0}^{p-1} x_i^3 = D$, $\sum_{i<j} x_ix_j = D(D-1)/2$, $\sum_{i,j} x_i^2x_j = D(D-1)$. We get:
\[
\begin{array}{rl}
\rho_3 =\displaystyle \sup_{\lambda, \{ \sigma_i^+,\sigma_i^-\} }  & \lambda \\
\textnormal{subject to} & \displaystyle  f - \lambda = \sigma_0 + \sum_{i=1}^m \sigma_i^+X_i + \sum_{i=1}^m \sigma_i^-(1-X_i)+ c_{X_0}(\sum_{i=0}^{p-1} X_i - D) \\
& +c_{X_0^2}(\sum_{i=0}^{p-1} X_i^2 - D)+c_{X_0^3}(\sum_{i=0}^{p-1} X_i^3 - D) \\
&+c_{X_0X_1}(\sum_{i<j} X_iX_j - D(D-1)/2)\\
&+c_{X_0^2X_1}(\sum_{i,j} X_i^2X_j - D(D-1))\\
 , \\
& \displaystyle  \sigma_0 \textnormal{ is a sum of squares of order } 2, \\
& \displaystyle  \sigma_i^+ \textnormal{ is a sum of squares of order } 2, \\
& \displaystyle  \sigma_i^- \textnormal{ is a sum of squares of order } 2, \\
&c_{X_0},c_{X_0^2},c_{X_0^3},c_{X_0X_1},c_{X_0^2X_1} \in \mathbb{R}.
\end{array}
\]
Let $v_1=v_1(X) = [1,X_0,\dots,X_{n-1}]^T$ and let us recall the notation $h(X)= \sum_\alpha [h(X)]_\alpha X^\alpha$. We can rewrite the relaxation as
\[
\begin{array}{rl}
\rho_3 = \displaystyle \sup_{\lambda, \{Q_i\}}  & \lambda \\
\textnormal{subject to} & [\displaystyle f(X) - v_1^TQ_0v_1 + \sum_{i=0}^{n-1} v_1^TQ_i^+v_1X_i + \sum_{i=0}^{n-1} v_1^TQ_i^-v_1(1-X_i)]_0 \\
&\displaystyle+c_{X_0}D +c_{X_0^2}D+c_{X_0^3}D+c_{X_0X_1}D(D-1)/2\\
&+c_{X_0^2X_1}D(D-1) = \lambda, \\
& \displaystyle [\displaystyle f(X) - v_1^TQ_0v_1 + \sum_{i=0}^{n-1} v_1^TQ_i^+v_1X_i + \sum_{i=0}^{n-1} v_1^TQ_i^-v_1(1-X_i)]_\alpha \\
&+c_\alpha= 0, \\
& \displaystyle \lambda \in \mathbb{R}, \\
& \displaystyle  Q_0,Q_i^+,Q_i^- \succeq 0 \textnormal{ for } i = 1,\dots,m. \\
\end{array}
\]
Because of the invariance under the affine group we can restrict the conditions to $\alpha \in \{ X_0, X_0^2, X_0X_1,X_0^3,X_0^2X_1,X^{s_1},\dots, X^{s_D}\}$, where $s_i \in S_i$ and $X^{s_t} = X_iX_jX_k$ if $s_t=\{i,j,k\}$. It is also enough to require that $Q_0,Q_1^+,Q_1^- \succeq 0$ as the other $Q$-matrices are just permutations of these. To solve this optimization problem for a fixed $p$ we need to do the following:

\begin{itemize}
\item[1.] Generate a basis of $\{0,1\}$-matrices $\{E_{1,1}$, $\dots$, $E_{1,d_1}$, $E_{2,1}$, $\dots$, $E_{2,d_2}$, $E_{3,1}$, $\dots$, $E_{3,d_3}\}$, where $\{ E_{1,1},\dots,E_{1,d_1} \}$ forms a $\{0,1\}$-basis for $Q_0$, $\{E_{2,1},\dots,E_{2,d_2}\}$ forms a $\{0,1\}$-basis for $Q_1^+$ and $\{E_{3,1},\dots,E_{3,d_3}\}$ forms a $\{0,1\}$-basis for $Q_1^-$. Let $B_{i,j}$ be the matrix $E_{i,j}$ normalized.
\item[1b.] (optional) Generate a basis $\{L_{1,1},\dots,L_{1,d_1},L_{2,1},\dots,L_{2,d_2},E_{3,1},\dots,L_{3,d_3}\}$
\item[2.] Find variables $C_{1,j,\alpha}$, measuring how much matrix $E_{1,j}$ contributes to the coefficient-part of the monomial $\alpha$ when multiplying out $v_1^TE_{1,j}v_1$. Similarly for $C_{2,j,\alpha}$ and $C_{3,j,\alpha}$, but have to calculate the total contribution of the sum $\sum_{k=0}^{n-1} v_{1,k}^TE_{2,j}v_{1,k}X_k$ and $\sum_{k=0}^{n-1} v_{1,k}^TE_{3,j}v_{1,k}(1-X_k)$ respectively where $v_{1,k}=[1,X_k,X_{k+1},\dots,X_{n-1},X_1,\dots,X_{k-1}]$. 
\item[3.] Solve the following optimization problem using a solver
\begin{verbatim}
variables g(1,1),...,g(1,d_1) g(2,1),...,g(2,d_2),
, g(3,1),...,g(3,d_3) lambda
maximize(lambda)
subject to
    
         sum_{i,j} C(i,j,1)*g(i,j)+c(X_0)D+c(X_0^2)D+...
         ...+c(X_0^3)D+...+c(X_0X_1)D(D-1)/2+...
         ...+c(X_0^2X_1)D(D-1) == lambda;
         
         sum_{i,j} C(i,j,X_0)*g(i,j) + c(X_0) == 0;
	
         sum_{i,j} C(i,j,X_0^2)*g(i,j)+ c(X_0^2) == 0;
	
         sum_{i,j} C(i,j,X_0X_1)*g(i,j)+ c(X_0X_1) == 0;
	
         sum_{i,j} C(i,j,X_0^3)*g(i,j)+ c(X_0^3) == 0;
	
         sum_{i,j} C(i,j,X_0^2X_1)*g(i,j)+ c(X_0^2X_1) == 0;
	
         sum_{i,j} C(i,j,X^{s_1})*g(i,j) == 0;
         .
         .
         .
         sum_{i,j} C(i,j,X^{s_D})*g(i,j) == 0;


         g(1,1)*B(1,1)+...+g(1,d_1)*B(1,d_1) is PSD
         g(2,1)*B(2,1)+...+g(2,d_2)*B(2,d_2) is PSD
         g(3,1)*B(3,1)+...+g(3,d_3)*B(3,d_3) is PSD
\end{verbatim}
Obtain the optimal value $\lambda$ and numerical certificates $Q_i=g_{i,1}B_{i,1}+...+g_{i,d_i}B_{i,d_i}$ as well as constants $c_\alpha$.
\end{itemize}

Each of the three steps takes a considerable amount of time to implement through hundreds of lines of codes. The same procedure is possible even for higher degree relaxation (the author has implemented a degree 5 relaxation for counting arithmetic progressions using the procedure above). Since the number of orbits $5$-sets gets partitioned into is much larger than the number of orbits $3$-sets gets partitioned into it gets extremely technical and difficult to write the code in part 1, 2 and 3. 

The degree 3 relaxation above was implemented in Matlab and the software CVX was used to solve the semidefinite program using the SeDuMi solver. The semidefinite program could easily be solved for all primes $p$ up to $300$ for all $D \in \{0,1/20,\dots,p\}$ on a macbook air, and then numerical difficulties started in the solver. The code was parallelized, so that it was solved for different values of $D$ at the same time, and uploaded to a super computer to get less accurate solutions for $p$ up to $613$. For a specific prime we were still able to calculate all values of $D$ in less than a day by using 100 computer nodes, but at this point the solutions were so inaccurate that solving for higher $p$ seemed pointless.

The degree 5 relaxation was similarly implemented, and could be used to find improved lower bounds for all primes $p \leq 19$. To solve for all $D$ at $p = 19$ took approximately 3 days using 100 computer nodes, and solving the problem for all $D$ at $p=23$ seemed impossible with the computer at hand. For further discussion on this problem we refer to \cite{Sjoland_density}.

\section*{Acknowledgements}
I would like to thank Alexander Engstr\"om for our discussions on how symmetries can be used in combinatorial problems. I also want to thank Markus Schweighofer and Cynthia Vinzant for their corrections and remarks.

\end{document}